\documentclass[11pt]{article}
\usepackage[utf8]{inputenc}
\usepackage{hyperref}
\usepackage{graphicx}
\usepackage{amsmath,amssymb}
\usepackage{amsthm}
\usepackage{soul}
\usepackage{tikz}
\usepackage{float}
\usepackage{cleveref}
\usetikzlibrary{arrows.meta}
\usepackage{rotating}
\usepackage[margin=2.5cm]{geometry}
\newtheorem{theorem}{Theorem}[section]     
\newtheorem{corollary}[theorem]{Corollary}

\newtheorem{proposition}[theorem]{Proposition}

\newtheorem{example}[theorem]{Example}

\newtheorem{problem}[theorem]{Problem}
\usepackage[style=numeric, sorting=ydnt]{biblatex}
\usepackage{fancyhdr}
\fancypagestyle{mypagestyle}{%
	\fancyhf{}
	\fancyhead[OC]{Subarsha Banerjee }
	\fancyhead[EC]{}
	\fancyfoot[C]{\thepage}%
}
\title{ON A NEW GRAPH DEFINED ON THE ORDER OF ELEMENTS OF A FINITE GROUP
	\footnote{Submitted to NSJOM on 15.05.2020 \&  accepted on 23.03.2021.}
}
\author{Subarsha Banerjee
	\footnote{The  author thanks the National Board of Higher Mathematics, Government of India for providing financial assistance.
	} 
	\\
	Department of Pure Mathematics, University of Calcutta\\
	35 Ballygunge Circular Road, Kol-700019\\
	West Bengal, India\\
	e-mail: subarshabnrj@gmail.com\\
	
}

\date{}

\begin{document}
	\maketitle
	
\begin{abstract}

In this paper, a new graph structure called the \textit{coprime order graph}  of a finite group $G$ denoted by $\Theta(G)$ has been introduced.
The \textit{coprime graph} of a finite group introduced by  Ma, Wei, and Yang [\textit{The coprime graph of a group. International Journal of Group Theory, 3(3), pp.13-23.}] is a subgraph of the \textit{coprime order graph} introduced in this paper. 
The vertex set of $\Theta(G)$ is $G$,  and any two vertices $x,y$ in $\Theta(G)$ are adjacent if and only if $\gcd(o(x),o(y))$ is equal to $1$ or a prime number.
We study how the  graph properties of $\Theta(G)$ and group properties of $G$ are related among themselves.
We provide a necessary and sufficient condition for  $\Theta(G)$ to be Eulerian for any finite group $G$. 
We also study  $\Theta(G)$ for certain finite groups like $\mathbb Z_n$ and $\mbox D_n$ and derive conditions when it is connected, complete, planar, and Hamiltonian for various $n\in \mathbb N$.
We also study the vertex connectivity of $\Theta(\mathbb Z_n)$ for various $n\in \mathbb N.$  
Finally, we have computed the signless Laplacian spectrum of $\Theta(G)$ when $G=\mathbb Z_n$ and $G=\mbox D_n$ for $n\in \{pq,p^m\}$ where $p,q$ are distinct primes and $m\in \mathbb{N}$.
\end{abstract}

\maketitle

\section{Introduction}

Generating graphs from various algebraic structures like groups and semigroups is nothing new.
Bosak in \cite{bosak1964graphs} studied various kinds of graphs that were defined on semigroups.
In \cite{zelinka1975intersection}, the author studied the \textit{intersection
graph} defined on a finite abelian group. A \textit{Cayley digraph} is also an important class of directed graphs defined on finite groups, and readers may refer to \cite{budden1985cayley,gallian2012contemporary} in order to find some information
about them. Kelarev and Quinn in \cite{kelarev2002directed} introduced the \textit{power graph} on  a semigroup $S$ as a directed graph in which the set of vertices is  $S$, and two distinct elements $a,b \in S$ are adjacent if and only if  $b=a^m$ for some positive integer $m$.
Motivated by the work in \cite{kelarev2002directed}, Chakrabarty \textit{et al.} studied the \textit{undirected power graph} on semigroups in \cite{chakrabarty2009undirected}. 
The undirected power graph on a semigroup $S$ is the graph  whose
vertex set is $S$, and two distinct vertices $a,b\in S$ are adjacent if and only if $a=b^m$ or $b=a^n$ for some positive integers $m, n$. Several  properties of power graph were investigated by Cameron and Ghosh in \cite{cameron2010power} and \cite{cameron2011power}.
In \cite{hamzeh2018order}, the authors  introduced a new graph known as the \textit{order supergraph} of the power graph of a finite group $G$, whose vertex set is $G$ and any two vertices $x, y$ are adjacent if and only if $o(x)\mid o(y)$ or $o(y)\mid o(x)$. The automorphism group of this graph was studied in \cite{hamzeh2017automorphism}.

Recently, several researchers have studied  spectral properties of  graphs associated with algebraic structures. 
The spectral properties of power graph of a finite group (\cite{mehranian2017spectra},\cite{hamzeh2017spectrum},\cite{banerjee2020signless},\cite{banerjee2021spectra}), Cayley graph of certain groups(\cite{babai1979spectra}, \cite{abdollahi2009cayley}, \cite{cheng2019integral}), commuting and non-commuting graph of dihedral groups(\cite{abdussakir2017spectra,banerjee2021metric}) etc. have been studied over the last few years.

The notion of \textit{coprime graph} of a finite group $G$ has existed in the literature for a long time. It was first introduced by Sattanathan and Kala as the \textit{order prime graph} in  \cite{sattanathan2009introduction}.
Later on, in \cite{ma2014coprime} Ma \textit{et al.} reintroduced and renamed the order prime graph as the \textit{coprime graph} and studied various properties of it.
The coprime graph was studied extensively in \cite{dorbidi2016note} and \cite{selvakumar2017classification}.
In \cite{banerjee2021laplacian}, the Laplacian spectra of coprime graph of finite cyclic and dihedral groups were studied.
In this paper, we introduce a new graph known as \textit{coprime order graph} of a finite group $G$. We denote it by $\Theta(G)$.
Clearly for a given finite group $G$, the \textit{coprime graph} is a subgraph of the \textit{coprime order graph} introduced in this paper.
We characterize some properties of  $\Theta(G)$ using the algebraic properties of the group $G$. We study the connectedness and the diameter of the graph $\Theta(G)$. We show that $\Theta(G)$ is Eulerian if and only if $G$ has odd order and every non-identity element of $G$ has prime order. We also find out when $\Theta(\mathbb{Z}_n)$ is planar
and Hamiltonian for various $n\in \mathbb{N}$.
We also study the vertex connectivity of $\Theta(\mathbb{Z}_n)$ for various $n$.
Finally, we find the signless Laplacian spectra of $\Theta(\mathbb{Z}_n)$  and $\Theta(\mbox D_n)$ for $n\in \{pq,p^m\}$ where $p,q$ are distinct primes and $m\in \mathbb{N}$.

The paper has been organized as follows:
In Section \ref{S2},  we have  provided the preliminary definitions and theorems that have been used throughout the paper.
In Section \ref{S3}, we formally introduce the \textit{coprime order graph} of a finite group $G$,  denoted by $\Theta(G)$,  and study  various properties of $\Theta(G)$.
In Section \ref{S4}, we study the vertex connectivity of $\Theta(\mathbb Z_n)$.
In Section \ref{S5}, we determine the signless Laplacian spectra of $\Theta(\mathbb{Z}_n)$  and $\Theta(\mbox D_n)$ for $n\in \{pq,p^m\}$.
\section{Preliminaries}
\label{S2}
In this section, for the convenience of the readers, we provide some preliminary definitions and theorems that have been used throughout the paper.
We denote a graph $\mathcal{G}$ by $\mathcal{G}=(V,E)$ where $V$ is the set of all vertices of $\mathcal{G}$ and $E$ denotes the set of all edges of $\mathcal{G}$.
A graph $\mathcal{G}$ is said to be \textit{simple} if it has no loops or parallel edges.
A graph with one vertex and no edges is called a \textit{trivial} graph. 
We denote the degree of a vertex $v\in V(\mathcal{G})$ by $\deg (v)$.
For a given graph $\mathcal{G}$, $\delta(\mathcal{G})=\min\{\deg(v):v\in \mathcal{G}\}$.
A \textit{subgraph} $\mathcal{H}=(W,F)$ of $\mathcal{G}=(V,E)$ is  a graph such that $W\subseteq V$ and $F\subseteq E$.
If there exists an edge between two vertices $a$ and $b$, then $a$ and $b$ are said to be adjacent, and it is denoted by $a\sim b$.
If there exists an edge between any two vertices of $\mathcal{G}$, then $\mathcal{G}$ is said to be \textit{complete} and is denoted by $K_n$.
A \textit{path} $P$ of length $k$ in a graph $G$ is an alternating sequence of vertices and edges $v_0, e_0, v_1, e_1, v_2, e_2, \ldots, v_{k-1}, e_{k-1}, v_k$, where $v_i's$  are distinct vertices, and $e_i$ is the edge joining $v_i$ and $v_{i+1}$. 
If $v_0=v_k$, then  $P$ is said to be a \textit{cycle} of length $k$.
The length of the shortest cycle in $\mathcal{G}$ is known as its \textit{girth}.
A graph  $\mathcal{G}$ is said to be  \textit{connected} if for any pair of vertices $u,v\in V$ there exists a path joining $u$ and $v$.
For a connected graph $\mathcal{G}$, the \textit{distance} between two vertices $u,v$ denoted by $d(u,v)$,  is defined as the length of the shortest path joining $u$ and $v$.
The \textit{diameter} of a connected graph $\mathcal{G}$, denoted by diam$(\mathcal{G})$, is  defined as diam$(\mathcal{G})=\max\{ d(u,v):u,v\in V\}$.
A \textit{planar} graph is a graph that can be embedded in the plane, i.e., it can be drawn on the plane in such a way that its edges intersect only at their endpoints.
An \textit{isomorphism} of graphs $\mathcal{G}$ and $\mathcal{H}$ denoted by $\mathcal{G}\cong \mathcal{H}$ is a bijection $f$ between  $V(\mathcal{G})$ and $V(\mathcal{H})$ such that any two vertices $u,v \in V(\mathcal{G})$ are adjacent if and only if the vertices $f(u),f(v)\in V(\mathcal{H})$ are adjacent.
An \textit{Eulerian cycle} in a graph $\mathcal{G}$ is a cycle which visits every edge exactly once.
A graph $\mathcal{G}$ is said to be \textit{Eulerian} if it has an \textit{Eulerian cycle}.
A \textit{Hamiltonian cycle} in a graph $\mathcal{G}$ is a cycle which visits every vertex exactly once.
A graph $\mathcal{G}$ is said to be \textit{Hamiltonian} if it has a \textit{Hamiltonian cycle}.
The \textit{vertex connectivity} $\kappa(\mathcal{G})$ of a graph $\mathcal{G}$ is the minimum number of vertices whose removal results in a disconnected or trivial graph. We define the  connectivity of a disconnected graph to be $0$.
Given a positive integer $k$, a graph $\mathcal{G}$ is said to be $k-$tough if for any integer $t>1$, $\mathcal{G}$ cannot be split into $t$ different connected components by the removal of fewer than $kt$ vertices.
The \textit{toughness} of a graph $\mathcal{G}$  is defined as the largest real number $t$ such that deletion of any $s$ vertices from $\mathcal{G}$ results in a graph which is either connected or else has at most $\frac{s}{t}$  components.
A \textit{dominating set} of a graph $\mathcal{G}$ is a subset $D$ of V such that for every $v\notin D$, there exists a vertex $w\in D$ for which $v$ is adjacent to $w$.
The \textit{domination number } is the number of vertices in a smallest dominating set of $\mathcal{G}$.
For more information on the terms used above, the readers may refer to any standard  book on graph theory, say \cite{diestel2005graph} or \cite{bondy1976graph}.

Let $\mathcal{G}$ be a finite simple undirected graph having vertex set $V(\mathcal{G})=\{v_1,v_2,\ldots,v_n\}$.
The \textit{adjacency matrix} of $\mathcal{G}$,  denoted by  $A(\mathcal{G})=(a_{ij})$ is defined as $a_{ij}=1$ if $v_i\sim v_j$ and $a_{ij}=0$ otherwise.
The \textit{degree matrix} of $\mathcal{G}$,  denoted by  $D(\mathcal{G})=(d_{ii})$ is a diagonal matrix,  where $d_{ii}$ denotes the degree of the $i^{th}$ vertex of $\mathcal{G}$.
The \textit{Laplacian matrix} $L(\mathcal{G})$ is defined as  $L(\mathcal{G}) = D(\mathcal{G})- A(\mathcal{G})$. 
The \textit{signless Laplacian matrix} $Q(\mathcal{G})$ is defined as  $Q(\mathcal{G}) = D(\mathcal{G})+ A(\mathcal{G})$.
The matrix $Q(\mathcal{G})$ is a real and symmetric matrix and hence all its eigenvalues are real.
Also, $Q(\mathcal{G})$ is a positive semi-definite matrix and hence all its eigenvalues are non-negative. 
For more information on $Q(\mathcal{G})$, readers may refer to \cite{cvetkovic2009towards}, \cite{cvetkovic2010towards} and \cite{cvetkovic2010towards1}.
We arrange the  eigenvalues of $Q(\mathcal{G})$ as $ \lambda_1(\mathcal{G})\geq \lambda_2(\mathcal{G}) \ge \cdots \ge \lambda_n(\mathcal{G})$  in non-increasing order, and repeated according to their multiplicities.

For $n\in \mathbb N$, the number of positive integers that are less than or equal to $n$ and are relatively prime to $n$ is denoted by $\varphi(n)$.
The function $\varphi$ is known as \textit{euler’s phi function}.
We know that  a finite cyclic group of order $n$ is isomorphic to $(\mathbb Z_n,+)$, where $\mathbb Z_n = \{0,1,2\ldots,n-2,n-1\}$, and hence we prove our results for $\mathbb{Z}_n$ instead of an arbitrary cyclic group.
An element $a\in \mathbb Z_n$ is said to be a \textit{generator} of $\mathbb Z_n$ if $\gcd(a,n)=1$.
An element which is not a generator is known as a \textit{non-generator}.
We denote the \textit{dihedral group} of order $2n$ by $\mbox{D}_n$.
The \textit{order} of an element $g\in G$, denoted by $o(g)$, is the least positive integer $n$ such that $g^n=e$, where $e$ is the identity element of $G$.
The number of elements in a set $S$ is denoted by $|S|$.
For basic definitions and notations on group theory, the readers are referred to  \cite{dummit2004abstract}.

The theorems used in the paper have been listed below.
The proof of theorems \ref{Th1}, \ref{Th11}, \ref{Kura}, \ref{Ore} can be found in   \cite{bondy1976graph} or \cite{diestel2005graph}, while the proof of Theorem \ref{Useit} can be found in \cite{bapat2010graphs}.
 
\begin{theorem}\label{Th1}
A connected graph $\mathcal{G}$ has an Eulerian cycle if and only if  $\deg(v)$ is even for all $v\in \mathcal{G}$.
\end{theorem}

\begin{theorem}\label{Th11} For any graph $\mathcal{G}$, 
$\kappa(\mathcal{G})\le \delta(\mathcal{G})$.
\end{theorem}

\begin{theorem}\label{Kura} The complete graph  $K_5$ and  the complete bipartite graph $K_{3,3}$ are non-planar. 
\end{theorem}

\begin{theorem}[Ore]\label{Ore} Let $\mathcal{G}$ be a finite and simple graph with $n$ vertices where $n\ge 3$.
If $\deg (v)+\deg (w) \geq n$ for every pair of distinct non-adjacent vertices $v$ and $w$ of $\mathcal{G}$, then $\mathcal{G}$ is Hamiltonian.
\end{theorem}

\begin{theorem}\cite{chvatal1973tough}\label{Thtough}
If $\mathcal{G}$ is Hamiltonian, then $\mathcal{G}$ is $1$-tough. 
\end{theorem}

\begin{theorem}
	\label{Useit}
	If $J$ denotes the square matrix of order $n$ with all entries equal to one and $I$ denotes the identity matrix of order $n$ then the eigenvalues of $aI+bJ$ are $a$ with multiplicity $n-1$ and $a+nb$ with multiplicity $1$.
\end{theorem}

\section{Coprime Order Graph of a Finite Group}
\label{S3}
Let $G$ be a finite group such that $|G|>2$.
The  coprime order graph $\Theta(G)=(V,E)$ is defined  as follows:
The vertex set $V$ is the set  $G$, and any two distinct vertices $x,y$ are adjacent if and only if $\gcd(o(x),o(y))$ is equal to  $1$ or a prime number.
We now study some basic properties of $\Theta(G)$.

\begin{theorem}
	The graph  $\Theta(G)$ satisfies the following properties:
	\begin{enumerate}
		\item [(a).] The domination number of $\Theta(G)$ is $1$ and $\{e\}$ is a dominating set of $\Theta(G)$.
		\item [(b).] The set  $\{x\}$ is a dominating set of $\Theta(G)$  if and only if $o(x)$ is equal to $1$  or a prime number.
	\end{enumerate}
\end{theorem}

\begin{proof}
	
	\begin{enumerate}
		
		\item [(a).] Since $\gcd(o(a),o(e))=o(e)=1$, we find that $e$ is adjacent to $a$ for all  $a\in G$.
		Hence, the set $\{e\}$ is a dominating set of $\Theta(G)$, which implies that the domination number of $\Theta(G)$ is $1$.

		\item [(b).] Let $x\in G$.
		If $o(x)$ is equal to $1$  or a  prime number, then $\gcd(o(a),o(x))=o(x)$ which is equal to $1$  or a prime number for all  $a\in G$. Thus, $x$ is adjacent to $a$ for all $a\in G$, which implies that  $\{x\}$ is a dominating set of $\Theta(G)$.
		\\
		Conversely, let  $\{x\}$ be a dominating set of $\Theta(G)$.
		Assume the contrary that $o(x)$ is neither $1$ nor a prime  number.
		Then, $o(x)$ is composite which implies that $x\neq x^{-1}$.
		Thus, $\gcd(o(x),o(x^{-1}))=o(x)$, which is composite.
		Hence, $x$ is not adjacent to $x^{-1}$.
		Since $x\neq x^{-1}$, it contradicts the fact that $\{x\}$ is a dominating set of $\Theta(G)$.
		Hence, $o(x)$ is either equal to $1$ or a  prime number.
		
	\end{enumerate}
\end{proof}

\begin{theorem}\label{Th2}
The graph $\Theta(G)$ is connected and the diameter of $\Theta(G)$ is at most $2$.
\end{theorem}
\begin{proof}
Let $x,y\in \Theta(G)$.
If $\gcd(o(x),o(y))$ is equal to $1$ or  a prime number, then $x$ is adjacent to $y$, and we are done.
If $\gcd(o(x),o(y))$ is composite then $x$ and $y$ are not adjacent.
Consider the identity element $e$ of $G$.
Since $o(e)=1$, so $x$  and $y$ are both  adjacent to $e$.
Thus, we find that there always exists a path of length $2$ between any two non-adjacent vertices $x,y\in \Theta(G)$. 
Thus, $\Theta(G)$ is connected and  the diameter of $(\Theta(G))$ is at most $2$.
\end{proof}

\begin{theorem}
	If the girth of $\Theta(G)$ is finite, then it equals $3$.
\end{theorem}

\begin{proof}
The proof follows from the simple fact that for any two distinct vertices $x, y\in G$ where $x,y\neq e$, there  exists a path of length $2$ given by $x\sim e\sim y$ from $x$ to $y$.
If $x,y$ are adjacent for some $x$ and $y$, then the  girth of $\Theta(G)$ is finite and it equals  $3$.
\end{proof}

\begin{theorem}
The graph $\Theta(G)$ is Eulerian if and only if $G$ is an odd-order group, and every non-identity element has prime order.
\end{theorem}

\begin{proof}
Suppose the graph $\Theta(G)$ is Eulerian.
Using Theorem \ref{Th2}, we find that $\Theta(G)$ is connected.
Since $\Theta(G)$ is Eulerian, using Theorem \ref{Th1} we find that every vertex in $\Theta(G)$ has an even degree.
Since the identity element $e$ of $G$ is connected to every other vertex in $\Theta(G)$, $\deg(e)=|G|-1$.
Since $\deg(e)$ must be  even, the order of $G$ must be odd. Thus, $G$ has no elements of order $2$.
Now let $a$ be any non-identity element of $G$. We claim that $a$ has prime order.
Assume that the order of $a$ is composite.
Let us consider the set 
\begin{align*}
E_a=\{b\in G:  \gcd(o(a),o(b)) \text{ is equal to } 1 \text{ or a prime number}\}.
\end{align*}
We notice that $b\in E_a$ if and only if $b^{-1}\in E_a$.
Thus, the number of non-identity elements present in $E_a$ (if any) are even.
Also, the identity element $e$ of $G$ is in $E_a$.
Thus,  $E_a$ has an odd number of elements.
Thus, $|E_a|$ is an odd number.
Let $E_a^{*}=E_a\setminus \{a\}$. Since the order of $a$ is composite, so $a\notin E_a$.
Thus, $E_a=E_a^{*}$.
We further note that the elements of  $E_a^{*}$ are those vertices of $\Theta(G)$ which are adjacent to the vertex $a$ of $\Theta(G)$.
Thus,  $|E_a|=|E_a^{*}|=\deg(a)$. 
Since  $\Theta(G)$ is Eulerian, $\deg(a)$ must be even, but we have proved that $\deg(a)$ is odd, which is contradictory.
Hence, our initial assumption that order of $a$ is composite, is false.
Thus, the order of $a$ must be a prime number.
Thus, we find that if $\Theta(G)$ is Eulerian, then the  order of $G$ is odd and every non-identity element has prime order.

Conversely, assume that $|G|$ is odd and every non-identity element of $G$ has  prime order.
Thus, for any element  $a\in  G$, we have   ${E_a}^*=G\setminus\{a\}.$
Also,
\begin{equation}\label{Eq5}
\deg(a)=|E_a^*|=|G\setminus\{a\}|.
\end{equation}
Since $|G|$ is an odd number, so $|G\setminus\{a\}|$ is an even number  for every $a\in G$.
Using Equation(\ref{Eq5}), we find that for every $a\in G$, $\deg(a)$ must be an even number.
Thus, $\Theta(G)$ is connected and every vertex in $\Theta(G)$ has  an even degree.
Using  Theorem \ref{Th1}, we conclude  that $\Theta(G)$ is Eulerian. 
Thus, the result follows.

\end{proof}

\begin{theorem}
	\label{Theorem1}
	The graph $\Theta(G)$ is complete if and only if $G$ has no elements of composite order.
\end{theorem}

\begin{proof}
	Suppose $\Theta(G)$ is complete.
	Let $g\in G$ be an element of composite order. Clearly $g\neq g^{-1}$.
	Since $\gcd(o(g),o(g^{-1}))=o(g)$, we find that $g$ is not adjacent to $g^{-1}$ in $\Theta(G)$.
	Thus, $\Theta(G)$ is not complete which is a contradiction.
	Hence, we conclude that $G$ has no element whose order is composite.
	Conversely, if all elements of $G$ have prime order, then  for any two elements $x,y \in G$, $\gcd(o(x),o(y))$ is equal to $1$ or a prime number, which in turn implies that $\Theta(G)$ is complete. 
\end{proof}

\begin{corollary}\label{Th3}
Let $G$ be a finite cyclic group of order $n$.
Then, $\Theta(G)$ is complete if and only if $n$ is a prime number.
\end{corollary}

\begin{corollary}
	\label{product}
Let $G$ be a finite commutative group  of order $p^m$ where $p$ is a prime and $m>1$.
Then, $\Theta(G)$ is complete if and only if $G\cong (\mathbb Z_p)^m$.
\end{corollary}

\begin{proof}
We know that any finite commutative group is a direct product of cyclic groups.
Hence, $G\cong \mathbb Z_{p^{\alpha_1}}\times \mathbb Z_{p^{\alpha_2}}\times \cdots \times \mathbb Z_{p^{\alpha_k}}$
where $\alpha_1+\alpha_2+\cdots+\alpha_k=m$ and $1\le \alpha_i\le m$.
Assume that $\Theta(G)$ is complete.
Now if $\alpha_i=1$ for all $1\le i\le k$, we are done.
So, let us assume that there exists  $\alpha_i$ such that $\alpha_i>1$ for some $i$.
Since $\mathbb Z_{p^{\alpha_i}}$ is a cyclic group of order $p^{\alpha_i}$, it has $\varphi(p^{\alpha_i})\ge 2$ generators, and hence we can find $x\in \mathbb Z_{p^{\alpha_i}}$ such that $o(x)=p^{\alpha_i}$.
Consider the element $\mathbf{x}=(0,0,\dots,0,x,0\dots,0,0)\in G$.
Clearly $o(\mathbf{x})$ is composite which contradicts Theorem \ref{Theorem1}. Hence, $\alpha_i=1$ for all $1\le i\le m$ which implies that $G\cong (\mathbb{Z}_p)^m.$
 The converse part is trivial and hence skipped.
\end{proof}

\begin{corollary}\label{Th31}
The graph  $\Theta(\mbox D_n)$  for $n\ge 2$ is complete  if and only if $n$ is prime.
\end{corollary}

\begin{proof}
We know that the dihedral group of order $2n$ has the following presentation:
$$\mbox{D}_n=\{\langle r,s\rangle : r^n=s^2=1,rs=sr^{-1}\}.$$
We partition the vertex set of $\Theta(\mbox D_n)$ as $\mbox D_n=A\cup B$ where  $A=\{r^i:0\le i\le n-1\}$,  and $B=\{sr^i:0\le i\leq n-1\}$.
The graph induced by the elements of $A$ forms a subgraph of $\Theta(\mbox D_n)$, and is isomorphic to $\Theta(\mathbb Z_n)$.
Since every element of $B$  has order $2$, so the subgraph induced by the elements  of $B$ is isomorphic to $K_n$.
Also since the order of each member of $B$ is $2$, every vertex of  $A$ is adjacent to every vertex of $B$. Using the above facts, we observe that $\Theta(\mbox D_n)$ is complete if and only if $\Theta(\mathbb Z_n)$ is complete.
Using Corollary \ref{Th3}, $\Theta(\mathbb Z_n)$ is complete if and only if $n$ is prime.
Thus, $\Theta(\mbox D_n)$ is complete if and only if $n$ is prime.

\end{proof}

\begin{theorem}
The graph $\Theta(\mathbb Z_n)$ is planar if and only if $n=3$ or $n=2^i$ where $i\in \mathbb{N}$.
\end{theorem}
\begin{proof}
Assume that $\Theta(\mathbb Z_n)$ is planar.
Let $n=p_1^{\alpha_1}p_2^{\alpha_2}\cdots p_k^{\alpha_k}$ where $p_i$'s are primes, and $\alpha_i$'s are positive integers.
Assume that  $\alpha_i \ge 1 $ for   $i=i_0$ and $i=i_1$,  then we can choose the elements $p_1^{\alpha_1}p_2^{\alpha_2}\cdots p_{i_0}^{\alpha_{i_0}-1} \cdots p_k^{\alpha_k}$ and $p_1^{\alpha_1}p_2^{\alpha_2}\cdots p_{i_1}^{\alpha_{i_1}-1} \cdots p_k^{\alpha_k}$, and consider the subgroup generated by these two  elements.
Every element in the subgroup generated by $p_1^{\alpha_1}p_2^{\alpha_2}\cdots p_{i_0}^{\alpha_{i_0}-1} \cdots p_k^{\alpha_k}$ except the identity element has order $p_{i_0}$, which is prime.
Every element in the subgroup generated by $p_1^{\alpha_1}p_2^{\alpha_2}\cdots p_{i_1}^{\alpha_{i_1}-1} \cdots p_k^{\alpha_k}$ except the identity element has order $p_{i_1}$, which is also prime.
Thus, we obtain $p_{i_0}+p_{i_1}-2$ elements of prime order.
Note that $ p_{i_0}+p_{i_1}\ge 5$ is always true, and hence we can always get at least $3$ elements of prime order.
If we take  $3$ elements of prime order, together with the zero element and a generator of $\mathbb Z_n$,  then we can find  $5$ elements that are adjacent to each other.
Thus, there always exists  a subgraph in $\Theta(\mathbb Z_n)$ which is isomorphic to $K_5$.
By  Theorem \ref{Kura}, we conclude that $\Theta(\mathbb Z_n)$ is not planar, which is contrary to our assumption.
Hence, we cannot find  $i_0$ and $i_1$ such that $\alpha_{i_0}, \alpha_{i_1}\ge 1$.
Thus, $\alpha_i\ge 1$ for at most one $i$.
Hence, $n=p^i$ where $p$ is a prime and $i\ge 1$.

Again if $p\ge 5,$ we can consider the element $p^{i-1}$.
We again notice that all the elements in the set $\{p^{i-1},2p^{i-1},3p^{i-1},\ldots, (p-1)p^{i-1}\}$  have prime order, and hence are adjacent to all other members of the graph $\Theta(\mathbb Z_n)$.
Since $p\ge 5$, so we have $p-1\ge 4$.
Hence, if we take  $4$  elements  from the set $\{p^{i-1},2p^{i-1},3p^{i-1},\ldots, (p-1)p^{i-1}\}$ together with the zero element of $\mathbb Z_n$, then the graph induced by them is isomorphic to $K_5$, and hence by Theorem \ref{Kura} we find that $\Theta(\mathbb Z_n)$ is not planar. Thus, we are left with primes $p=2,3$.
Hence, either $n=2^i$ or $n=3^i$ for some $i\in \mathbb{N}$.
We claim that the graph $\Theta(\mathbb Z_n)$ is planar when  $n=2^i$ for all $i\ge 1$.
If $n=2^i$ for some $i$, then  $\Theta(\mathbb Z_n)$ has exactly two vertices of degree $n-1$, and the remaining vertices will each have degree $2$.
We illustrate  $\Theta(\mathbb Z_n)$ for $n=8$ below.
The graph $\Theta(\mathbb Z_8)$ can be suitably drawn as the following:
\begin{figure}[H]
	\centering
	\begin{tikzpicture}
   
    \node[shape=circle,draw=black] (0) at (0,0) {$0$};  
    \node[shape=circle,draw=black] (4) at (6,0)  {$4$};   
    \node[shape=circle,draw=black] (1) at  (3,-1) {$1$};  
    \node[shape=circle,draw=black] (2) at  (3,-2) {$2$};  
    \node[shape=circle,draw=black] (3) at (3,-3) {$3$}; 
    \node[shape=circle,draw=black] (5) at (3,1) {$5$}; 
    \node[shape=circle,draw=black] (6) at (3,2) {$6$}; 
    \node[shape=circle,draw=black] (7) at (3,3) {$7$};  
    \draw (0) -- (1);  
    \draw (0) -- (2);  
    \draw (0) -- (3);  
    \draw (0) -- (4); 
    \draw (0) -- (5);  
    \draw (0) -- (6);  
    \draw (0) -- (7);  
    \draw (4) -- (1);  
    \draw (4) -- (2);  
    \draw (4) -- (3);  
    \draw (4) -- (5);  
    \draw (4) -- (6);  
    \draw (4) -- (7);
    
    \end{tikzpicture}
    \caption{$\Theta(\mathbb Z_8)$}
    \label{Fig1}
\end{figure}
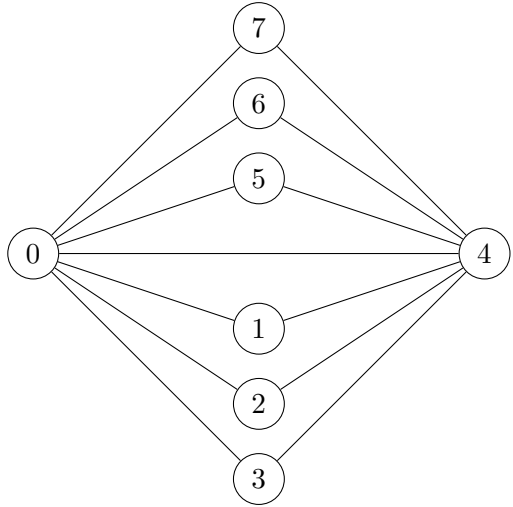
From Figure \ref{Fig1} it is evident that  $\Theta(\mathbb Z_8)$ is planar.
Using the same arguments as done for  $\Theta(\mathbb Z_8)$, it can be established that    $\Theta(\mathbb Z_n)$ is planar for $n=2^i$ where $i\in \mathbb{N}$.

Now we show that $\Theta(\mathbb Z_n)$ is planar for $n=3^i$ only for $i=1$.
Note that when $i=1$, then $\Theta(\mathbb Z_3)\cong K_3$ which is planar.  
We illustrate $\Theta(\mathbb Z_9)$ below:
\begin{figure}[H]
	\centering
	\begin{tikzpicture}
	\node[shape=circle,draw=black] (0) at (0,0) {$0$};  
	\node[shape=circle,draw=black] (3) at (-5,-1.5)  {$3$};   
	\node[shape=circle,draw=black] (6) at  (0,-4) {$6$};  
	\node[shape=circle,draw=black] (1) at  (3,0) {$1$};  
	\node[shape=circle,draw=black] (2) at (3,-2) {$2$}; 
	\node[shape=circle,draw=black] (5) at (3,-4) {$5$}; 
	\node[shape=circle,draw=black] (4) at (-3,2) {$4$}; 
	\node[shape=circle,draw=black] (7) at (-4,-3) {$7$};  
	\node[shape=circle,draw=black] (8) at (3,2) {$8$};  
	\draw (0) -- (1);  
	\draw (0) -- (2);  
	\draw (0) -- (3);  
	\draw (0) -- (4); 
	\draw (0) -- (5);  
	\draw (0) -- (6);  
	\draw (0) -- (7); 
	\draw (0) -- (8);  
	
	\draw (3) -- (1);  
	\draw (3) -- (2);  
	\draw (3) -- (4); 
	\draw (3) -- (5);  
	\draw (3) -- (6);  
	\draw (3) -- (7); 
	\draw (3) -- (8);
	
	\draw (6) -- (1);  
	\draw (6) -- (2);  
	\draw (6) -- (4); 
	\draw (6) -- (5);  
	 
	\draw (6) -- (7); 
	\draw (6) -- (8);
	\end{tikzpicture}
	\caption{$\Theta(\mathbb Z_9)$}
	\label{Fig2}

\end{figure}
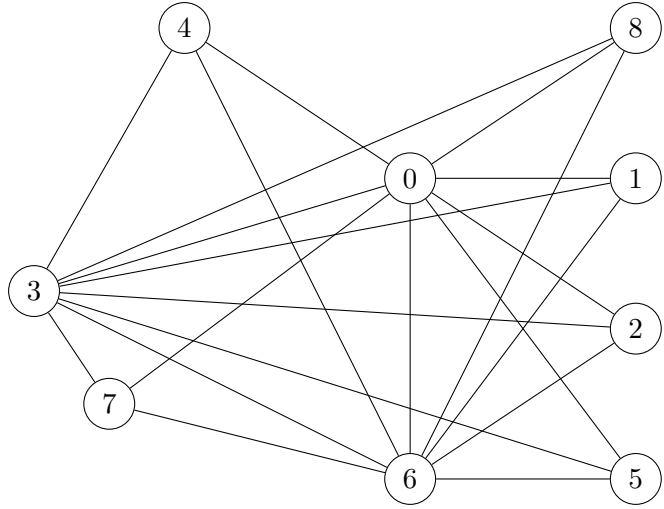

Now for $i\ge 2$, if we take the vertices $0,3^{i-1},2\cdot 3^{i-1}$, and any three vertices other than these, then the graph obtained contains $K_{3,3}$ as a subgraph. Using Theorem \ref{Kura}, we can conclude that $\Theta(\mathbb Z_n)$ is not planar for $n=3^i$ where $i\ge 2$. Hence, the graph is planar only when $n=3$.
Thus, $\Theta(\mathbb Z_{n})$ is planar if and only if  $n=3$ or $n=2^i$ where $i\in \mathbb{N}$.
\end{proof}
\begin{theorem}
If $p$ and $q$ are distinct primes with $p<q$ then  $\Theta(\mathbb Z_{pq})$ is Hamiltonian if and only if $p=2$.
 
\end{theorem}
\begin{proof}
Assume that $p=2$.
Let $v_1$ and $v_2$ be two non-adjacent vertices in  $\Theta(\mathbb Z_{2q})$. 
Then, $v_1$ and $v_2$ are generators of $\mathbb Z_{2q}$.
Note that $v_1$ is adjacent to any non-generator of $\mathbb Z_{2q}$ and so is $v_2$.
Then, $\deg(v_i)=2q-\varphi(2q)=2q-(q-1)=q+1$ where $i\in \{1,2\}$.
Thus, $\deg(v_1)+\deg(v_2)= 2(q+1)>2q$.
Thus, the sum of degrees of two non-adjacent vertices is greater than the number of vertices in $\Theta(\mathbb Z_{2q})$.
By Theorem \ref{Ore}, we conclude that $\Theta(\mathbb Z_{2q})$ is Hamiltonian.

Now we show that if $2<p<q$ then the graph $\Theta(\mathbb Z_{2q})$ is not Hamiltonian.
Consider the sets $$A=\{i:\gcd(i,n)=1\} \text{ and }B=\{0\}\cup \{i:\gcd(i,n)\neq 1\}.$$
If we remove all vertices of $\Theta(\mathbb Z_{2q})$ which are in $B$ then  $\Theta(\mathbb Z_{2q})$ has $|A|=\varphi(pq)$ components.
Since $2<p<q$, we obtain,
\begin{equation}\label{tough}
\begin{split}
(p-2)(q-2)> 2
&\implies  pq-2p-2q+2> 0\\
&\implies  pq+2> 2(p+q)\\
&\implies  pq-p-q+1>p+q-1 \\
&\implies  (p-1)(q-1)>p+q-1\\
&\implies  \varphi(pq)>pq-\varphi(pq)\\
&\implies  |A|>|B|.
\end{split}
\end{equation}

Using  Equation (\ref{tough}) we find that $\Theta(\mathbb Z_{pq})$ is not $1$-tough.
Using Theorem \ref{Thtough} we conclude that $\Theta(\mathbb Z_{pq})$ is not Hamiltonian when $2<p<q$.
Thus,  $\Theta(\mathbb Z_{pq})$ is Hamiltonian if and only if $p=2$.
 
\end{proof}

We observe that if two groups $G_1$ and $G_2$ are isomorphic, then the corresponding graphs $\Theta(G_1)$ and $\Theta(G_2)$ are isomorphic to each other.
However, the converse is false. To illustrate it we consider the following example:

\begin{example}
Consider the unitriangular matrix group $\mathfrak F=\biggl\{\left(
\begin{array}{cccccccc}
1& a&b\\0&1&c\\0&0&1
\end{array}
\right): a, b, c\in \mathbb F_3\biggr \}$ where $\mathbb F_3$ denotes the finite field of order $3$.
Clearly $\mathfrak F$ forms a group under matrix multiplication.
Now consider the group $(\mathbb Z_3)^3$. Using Corollary \ref{product}, $\Theta((\mathbb Z_3)^3)$ is complete.
We also notice that each non-identity element of the group $\mathfrak F$ has order $3$.
Thus, any two elements of $\Theta(\mathfrak{F})$ are adjacent to each other which in turn implies that $\Theta(\mathfrak F)$  is complete.
Since $\mathfrak{F}$ is non-commutative whereas $(\mathbb Z_3)^3$ is commutative, we find that the two groups  are not isomorphic to each other.
However, $\Theta(\mathfrak{F})$ and $\Theta((\mathbb Z_3)^3)$ are isomorphic to each other as both are complete graphs having $27$ elements.
	
\end{example}

\section{Vertex Connectivity of $\Theta(G)$ }
\label{S4}
In this section, we first investigate the vertex connectivity of $\Theta(\mathbb Z_n)$ for $n\ge 2$. For a given group $G$, we fix the following notations:
Let $S^{*}(G)$ denote  set of all those elements  $G$ which have prime order.
Let $S(G)=\{e\}\cup S^{*}(G)$ where $e$ denotes the identity element of $G$.

\begin{proposition}\label{Th10}
If $n$ is  prime, then $\kappa(\Theta(\mathbb Z_n))=n-1.$
\end{proposition}

\begin{proof}
Since $n$ is prime,  $\Theta(\mathbb Z_n)$ is complete(Corollary \ref{Th3}).
Since  vertex connectivity of a complete graph on $n$ vertices is $n-1$, we conclude that $\kappa(\Theta(\mathbb Z_n))=n-1.$
\end{proof}

\begin{theorem}\label{Th12}
If $n$ is composite, then $\kappa(\Theta(\mathbb Z_n))= |S(\mathbb Z_n)|.$
\end{theorem}
\begin{proof}
Let  $v_0$ be a generator of $\mathbb Z_n$. Thus, $o(v_0)=n$.
Since $v_0$ is a generator so $\deg(v_0)\le \deg (w)$ for all vertices $w\in \Theta(\mathbb Z_n)$ which implies that $\delta(\Theta(\mathbb Z_n))=\deg(v_0)$.
Now we notice that the vertex $v_0\in \Theta(\mathbb Z_n)$ is adjacent only to all the elements of $S(\mathbb Z_n)$ and nothing else.
Thus, $\deg(v_0)=|S(\mathbb Z_n)|$.
By Theorem \ref{Th11}, $\kappa(\Theta(\mathbb Z_n))\le |S(\mathbb Z_n)|.$
Now we claim  that  $S(\mathbb{Z}_n)$ is a minimum separating set of $\Theta(\mathbb Z_n)$. If not, then suppose we remove $|S(\mathbb{Z}_n)|-1$ elements from the vertex set of $\Theta(\mathbb Z_n)$. Then, there exists $a\in S(\mathbb{Z}_n)$ such that $a$ is adjacent to all other vertices of $\Theta(\mathbb Z_n)$, making $\Theta(\mathbb Z_n)$  connected.
Thus, $S(\mathbb{Z}_n)$ is a minimum separating set of $\Theta(\mathbb Z_n)$, which proves the fact that $\kappa(\Theta(\mathbb Z_n))= |S(\mathbb Z_n)|.$
\end{proof}

\begin{corollary}\label{VC}
If $n=pq$ where $p,q$ are distinct primes with $p<q$, then $\kappa(\Theta(\mathbb Z_n))=p+q-1$.
\end{corollary}

\begin{proof}
If $n=pq$, then $S(\mathbb Z_{pq})=\{0,p,2p,3p,\ldots,(q-1)p,q,2q,3q,\ldots ,(p-1)q\}.$
Since  $|S(\mathbb Z_{pq})|=p+q-1,$ the result follows.
\end{proof}

\begin{corollary}
If $n=p^m$ where $p$ is a  prime  and $m\in \mathbb N$ then $\kappa(\Theta(\mathbb Z_n))=p$.
\end{corollary}
\begin{proof}

If $n=p^m$, then $S(\mathbb Z_{p^m})=\{0,p^{m-1},2p^{m-1},3p^{m-1},\ldots,(p-1)p^{m-1}\}.$
Since  $|S(\mathbb Z_{p^m})|=p,$ the result follows.
\end{proof}

Now it is quite natural to ask that if $\Theta(G)$ is not complete, is it true that $\kappa(\Theta(G))$ equals $|S(G)|$? We show that it is false.
Consider the Dicyclic group $\mbox{Dic}_n$ of order $4n$ given by :
\begin{flalign*}
\mbox{Dic}_n=\{\langle a,x\rangle:a^{2n}=1, x^2=a^n, ax=xa^{-1}\}.
\end{flalign*}
We illustrate $\Theta(\mbox{Dic}_3)$ below:
\begin{figure}[H]
	\centering
	\begin{tikzpicture}
		\node[shape=circle,draw=black] (1) at (-5,2) {$1$};  
	\node[shape=circle,draw=black] (a) at (-3,3)  {$a$};   
	\node[shape=circle,draw=black] (a2) at  (-6,1) {$a^2$};  
	\node[shape=circle,draw=black] (a3) at  (-5,-4) {$a^3$};  
	\node[shape=circle,draw=black] (a4) at (-6,-1) {$a^4$}; 
	\node[shape=circle,draw=black] (a5) at (-3,-5) {$a^5$};

	\node[shape=circle,draw=black] (x) at (4,2) {$x$};  
	\node[shape=circle,draw=black] (xa) at (3,3)  {$xa$};   
	\node[shape=circle,draw=black] (xa2) at  (3,0) {$xa^2$};  
	\node[shape=circle,draw=black] (xa3) at  (4,-2) {$xa^3$};  
	\node[shape=circle,draw=black] (xa4) at (3.5,-3.5) {$xa^4$}; 
	\node[shape=circle,draw=black] (xa5) at (3,-5) {$xa^5$};

	\draw (1) -- (a);  
	\draw (1) -- (a2);  
	\draw (1) -- (a3); 
	\draw (1) -- (a4);  
	\draw (1) -- (a5);  
	\draw (1) -- (x); 
	\draw (1) -- (xa);  
	\draw (1) -- (xa2);  
	\draw (1) -- (xa3);  
	\draw (1) -- (xa4);  
	\draw (1) -- (xa5);

	\draw (a2) -- (a);    
	\draw (a2) -- (a3); 
	\draw (a2) -- (a4);  
	\draw (a2) -- (a5);  
	\draw (a2) -- (x); 
	\draw (a2) -- (xa);  
	\draw (a2) -- (xa2);  
	\draw (a2) -- (xa3);  
	\draw (a2) -- (xa4);  
	\draw (a2) -- (xa5);

	\draw (a3) -- (a);    
	\draw (a3) -- (a4);  
	\draw (a3) -- (a5);  
	\draw (a3) -- (x); 
	\draw (a3) -- (xa);  
	\draw (a3) -- (xa2);  
	\draw (a3) -- (xa3);  
	\draw (a3) -- (xa4);  
	\draw (a3) -- (xa5);

	\draw (a4) -- (a);    
	\draw (a4) -- (a3); 
	\draw (a4) -- (a4);  
	\draw (a4) -- (a5);  
	\draw (a4) -- (x); 
	\draw (a4) -- (xa);  
	\draw (a4) -- (xa2);  
	\draw (a4) -- (xa3);  
	\draw (a4) -- (xa4);  
	\draw (a4) -- (xa5);
	
	\draw (a) -- (x); 
	\draw (a) -- (xa);  
	\draw (a) -- (xa2);  
	\draw (a) -- (xa3);  
	\draw (a) -- (xa4);  
	\draw (a) -- (xa5);
	
	\draw (a5) -- (x); 
	\draw (a5) -- (xa);  
	\draw (a5) -- (xa2);  
	\draw (a5) -- (xa3);  
	\draw (a5) -- (xa4);  
	\draw (a5) -- (xa5);

	\end{tikzpicture}
	\caption{$\Theta(\mbox {Dic}_3)$}
	\label{Fig3}

\end{figure}
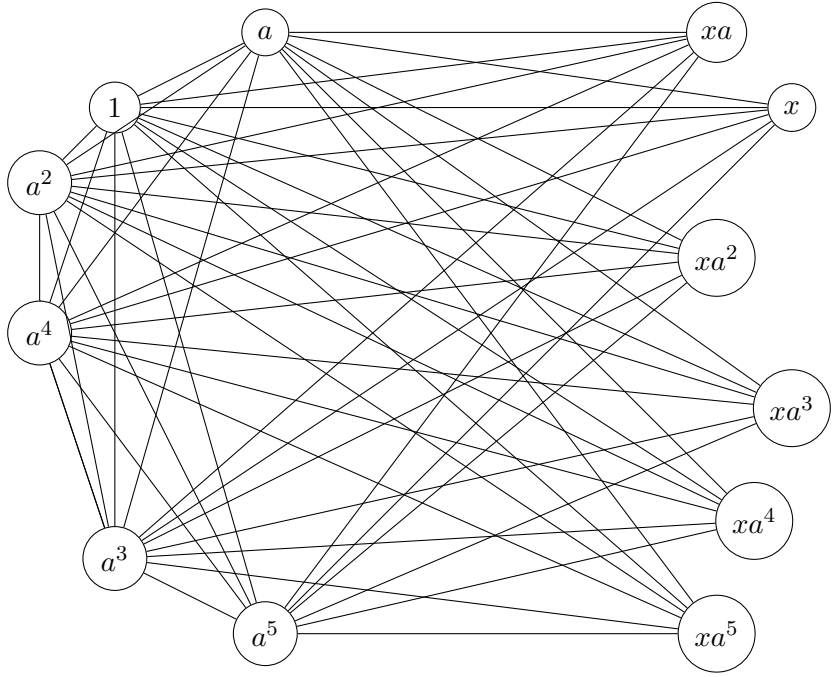
From Figure \ref{Fig3}, we observe that  $\{1,a,a^2,a^3,a^4,a^5\}$ is a minimum separating set of $\Theta(\mbox {Dic}_3)$.
So, $\kappa(\Theta(\mbox {Dic}_3))=6$.
Also, $S(\mbox{Dic}_n)=\{1,a^2,a^3,a^4\}$.
Hence, $\kappa(\Theta(\mbox {Dic}_3))>|S(\mbox{Dic}_3)|$.

We thus end this section by proposing  the following open problem which can be considered for further research.
\begin{problem}
Characterize all finite groups $G$, such that $\Theta(G)$ is not complete, but $\kappa(\Theta(G))=|S(G)|$.
\end{problem}

\section{Signless Laplacian Spectrum of $\Theta(G)$}
\label{S5}
In this section, we shall find the signless Laplacian spectra of $\Theta(\mathbb{Z}_n)$  and $\Theta(\mbox D_n)$ for $n\in \{pq, p^m\}$ where $p,q$ are distinct primes with $p<q$ and $m\in \mathbb{N}$.
We denote the signless Laplacian matrix of $\Theta(G)$  by $Q=Q(\Theta(G))$.

\subsection{Signless Laplacian Spectrum of $\Theta(\mathbb{Z}_n)$}

\begin{theorem}\label{Spec1}
If $n=pq$, then the eigenvalues of $Q(\Theta(\mathbb{Z}_n))$ are $p+q-1$ with multiplicity $pq-p-q$, $pq-2$ with multiplicity $p+q-2$ and other two are solutions of the equation $x^2-x(pq+2p+2q-4)+2(p+q-1)(p+q-2)=0$. 
\end{theorem}

\begin{proof}
The rows and columns of $Q=Q(\Theta(\mathbb{Z}_n))$ have been indexed in the following way:

We start with the zero element $0$ of $\mathbb Z_n$.
We then list those elements $m\in \mathbb Z_n$ such that $\gcd(m,n)\neq 1$.
Finally, we list those elements $m\in \mathbb Z_{n}$ such that $\gcd(m,n)= 1$.
Using the above indexing, $Q$  takes the following form,
\begin{equation}\label{Equation 1}
\begin{split}
Q&=\left(
\begin{array}{cccccccc}
((n-2)I+J)_{(n-\varphi(n))\times (n-\varphi(n))}&&&& J_{(n-\varphi(n))\times \varphi(n)}\\
\\
J^T_{\varphi(n)\times (n-\varphi(n))}&&&& (n-\varphi(n))I_{\varphi(n)\times \varphi(n)}
\end{array}
\right).
\end{split}
\end{equation}
Here, $J_{m\times n}$ is a matrix of order $m\times n$  all of whose entries are $1$.
\\
If we consider the matrix $Q-(n-2)I$ we obtain,
\begin{equation}\label{mat1}
Q-(n-2)I=\left(
\begin{array}{ccccccccccc}
J_{n-\varphi(n)} & & J_{(n-\varphi(n))\times \varphi(n)} 
\\
\\
J^T_{\varphi(n)\times (n-\varphi(n))}  & & (2-\varphi(n))I_{\varphi(n)}
\end{array}
\right).
\end{equation}

Since the matrix in Equation (\ref{mat1}) has $n-\varphi(n)$ identical rows, we conclude that $n-2$ is an eigenvalue of $Q$ with multiplicity at least $n-\varphi(n)-1$.
Similarly, if we consider the matrix $Q-(n-\varphi(n))I$ we find that it has $\varphi(n)$ identical rows, which makes us conclude that $n-\varphi(n)$ is an eigenvalue of $Q$ with multiplicity at least $\varphi(n)-1$.
We shall use the concept of  \textit{equitable partitions}(see \cite[Section $5$]{banerjee2020signless}) to find the remaining two eigenvalues of $Q$.
In short, given a graph $\mathcal{G}$, a partition $\pi $ of $V(\mathcal{G})=V_1\cup V_2\cup \cdots \cup V_k$ is an equitable partition of $\mathcal{G}$, if every vertex in $V_i$ has the same number of neighbors in $V_j$ for all $i,j\in \{1,2,\ldots ,k\}$.
Also, given an equitable partition $\pi$ of   $\mathcal{G}$, and its signless Laplacian matrix $Q$, we can form a matrix $Q_{\pi}=(q_{ij})$ in the following way
\begin{equation}
\label{EqP}
q_{ij}= 
\begin{cases} 
b_{ij} & \text{ if }i\neq j \\
b_{ii} +\sum _{j=1}^k b_{ij} & \text{ if } i=j\\ 
\end{cases}
\end{equation}
where $b_{ij}$ is the number of neighbors  a vertex $v\in V_i$ has in $V_j$, and $b_{ii}$ is the number of neighbors  a vertex $v\in V_i$ has in $V_i$.
We refer to the matrix $Q_{\pi}$ as the  matrix corresponding to the partition $\pi$ of  $\mathcal{G}$.
It is further  known  that the multiset of eigenvalues of $Q_{\pi}$ is contained in the multiset of eigenvalues of $Q$ \cite[ Lemma 5.1]{banerjee2020signless}.

We partition the vertex set $V$ of $\Theta(\mathbb{Z}_n)$ as $V_1\cup V_2$ where $V_1=\{0\}\cup \{m: \gcd(m,n)\neq 1\}$ and $V_2=V\setminus V_1$. 
We notice that each vertex $v$ in  $V_1$ has $n-\varphi(n)-1$ neighbors in $V_1$, and $\varphi(n)$ neighbors in $V_2$.
Also, each vertex $v$ in $V_2$ has $n-\varphi(n)$ neighbors in $V_1$, and $0$ neighbors in $V_2$.
We call this partition $\pi$.
Using Equation (\ref{EqP}), we can construct the \textit{equitable quotient} matrix $Q_{\pi}$ corresponding to this  partition $\pi$ of $\Theta(\mathbb{Z}_n)$.

\begin{equation*}
Q_{\pi}=\left(
\begin{array}{ccccc}
2n-\varphi(n)-2  && \varphi(n)
\\
\\
n-\varphi(n) && n-\varphi(n)
\end{array}
\right).
\end{equation*}

The characteristic polynomial of $Q_{\pi}$ is given by $$\Lambda(x)=x^2+x(2-3n+2\varphi(n))+(2n-2\varphi(n))(n-1-\varphi(n)).$$
The solutions of $\Lambda(x)=0$ are $\frac{1}{2}\{3n-2-2\varphi(n)\pm \sqrt{4n\varphi(n)-4\varphi(n)^2+n^2-4n+4}\}$.
Since $n>2$, so $\varphi(n)\ge 2$, hence we have
\begin{flalign*}
\Lambda(n-\phi(n))=\phi(n)(\phi(n)-n)\neq 0\\
\text{ and }\Lambda(n-2)=2(\phi(n)-1)(\phi(n)-n)\neq 0.
\end{flalign*}

Since the eigenvalues of $Q_{\pi}$ are different from $n-\varphi(n)$ and $n-2$, using   \cite[Lemma $5.1$]{banerjee2020signless} we find that the remaining eigenvalues of $Q$  are $\frac{1}{2}\{3n-2-2\varphi(n)\pm \sqrt{4n\varphi(n)-4\varphi(n)^2+n^2-4n+4}\}$.
Thus, the  eigenvalues of $Q$ are $n-\varphi(n)$ with multiplicity $\varphi(n)-1$, $n-2$ with multiplicity $n-\varphi(n)-1$, and other two are solutions of the equation $x^2+x(2-3n+2\varphi(n))+(2n-2\varphi(n))(n-1-\varphi(n))=0$.
On substituting $n=pq$, we find that the eigenvalues of $Q$ are $p+q-1$ with multiplicity $pq-p-q$, $pq-2$ with multiplicity $p+q-2$, and other two are solutions of the equation $x^2-x(pq+2p+2q-4)+2(p+q-1)(p+q-2)=0$, and hence the result follows.

\end{proof}

\begin{proposition}
	If $n=p$, then the eigenvalues of $Q(\Theta(\mathbb{Z}_n))$ are  $2(n-1)$ with multiplicity $1$ and $n-2$ with multiplicity $n-1$.
\end{proposition}

\begin{proof}
	If $n=p$, then using Corollary \ref{Th3}, $\Theta(\mathbb{Z}_n)$ is complete.
	Thus, $Q=(n-2)I +J$. Using Theorem \ref{Useit}, the eigenvalues of $Q$ are  $2(n-1)$ with multiplicity $1$ and $n-2$ with multiplicity $n-1$.
\end{proof}

\begin{theorem}\label{Spec2}

If $n=p^m$, where $m\ge 2$, then the eigenvalues of $Q(\Theta(\mathbb{Z}_n))$  are  $p$ with multiplicity $p^m-p-1$, $p^m-2$ with multiplicity $p-1$, and the other two are given by the solutions of the equation $x^2-x(p^m+2p-2)+2p(p-1)=0$.

\end{theorem}

\begin{proof}
 
The rows and columns of the matrix $Q$ have been indexed in the following way,
\\
We start with the zero element $0$ of $\mathbb Z_n$.
We then list the following elements of $\mathbb Z_n$,
 $$\{p^{m-1},2p^{m-1},3p^{m-1},\ldots ,(p-2)p^{m-1},(p-1)p^{m-1}\}.$$
We then list the remaining non-generators of $\mathbb Z_n$, and finally we list the generators of $\mathbb Z_n$.
Since each element of the set  $\{p^{m-1},2p^{m-1},3p^{m-1},\ldots,(p-2)p^{m-1},(p-1)p^{m-1}\}$ has order $p$ we find that they are adjacent to all other vertices of $\Theta(\mathbb{Z}_n)$.
Using the above indexing,  $Q$ takes the following form,
\begin{flalign*}
Q&=\left(
\begin{array}{ccccccccc}
((n-2)I+J)_{p\times p}&& J_{p\times (n-p)}
\\
\\
J^T_{(n-p)\times p}&& pI_{(n-p)\times (n-p)}
\end{array}
\right).
\end{flalign*}

If we consider the matrix $Q-(n-2)I$, we find that it has $p$ identical rows, and hence $n-2$ is an eigenvalue of $Q$ with multiplicity at least $p-1$.
Similarly, if we consider the  matrix $Q-pI$, we find that it has $p^m-p$ identical rows, and hence $ p$ is an eigenvalue of $Q$ with multiplicity at least $p^m-p-1$.

We again partition the vertex set $V$ of $\Theta(\mathbb{Z}_n)$ as $V=V_1\cup V_2$ where 
$V_1=\{0,p^{m-1},2p^{m-1},\ldots ,(p-2)p^{m-1},(p-1)p^{m-1}\}$ and $V_2=V\setminus V_1$. We observe that each vertex $v$ in $V_1$ has $p-1$ neighbors in $V_1$ and $p^m-p$ neighbors in $V_2$. Similarly, each vertex $v$ in $V_2$ has $p$ neighbors in $V_1$ and $0$ neighbors in $V_2$. Thus, the partition is an equitable partition, and we call the partition $\pi$.

According to Equation (\ref{EqP}), the equitable quotient matrix corresponding to the partition $\pi$ becomes,
\[ 
Q_{\pi}=\left(
\begin{array}{ccccccccc}
p^m+p-2 && p^m-p\\
p && p
\end{array}\right)
\]

The characteristic polynomial of $Q_{\pi}$ is $$\Lambda(x)=x^2-x(p^m+2p-2)+2p(p-1).$$	
Note that for a given  $p$ and $m\ge 2$,
\begin{flalign*}
\Lambda(p)=p(p-p^m)\neq 0\\
\text{ and } \Lambda(p^m-2)=-2p(p^m-p-1)\neq 0
\end{flalign*}
We find that the solutions of $\Lambda(x)=0$ are $\frac{1}{2}\biggl(p^m+2p-2\pm\sqrt{(p^m+2p-2)^2-8p(p-1)}\biggr)$.
Since the eigenvalues of $Q_{\pi}$ are different from $p$ and $p^m-2$, using   \cite[Lemma $5.1$]{banerjee2020signless} we find that the eigenvalues of $Q$ are   $p$ with multiplicity $p^m-p-1$, $p^m-2$ with multiplicity $p-1$, and the other two are solutions of the equation $x^2-x(p^m+2p-2)+2p(p-1)=0$.

\end{proof}

\subsection{Signless Laplacian Spectrum of $\Theta(\mbox D_{n})$}

In this section, we shall find the signless Laplacian spectrum of $\Theta(\mbox D_n)$.
We know that $$\mbox D_n=\{\langle r,s\rangle :r^n=s^2=1, rs=sr^{-1}\}.$$
We first index the elements $r^i$,  and then index the elements $sr^i$ where $0\le i\le n-1$.
We also note that $o(sr^i)=2$ for all $0\le i\le n-1$, and hence $\gcd(o(sr^i),o(sr^j))=2$ for all $0\le i,j\le n-1$.
Also, $r^i$ is adjacent to $sr^j$ for all $0\le i,j\le n-1$.
The signless Laplacian matrix of $\Theta(\mbox D_{n})$ is given by:

\begin{equation}\label{Di}
Q(\Theta(\mbox D_{n}))=
\left(
\begin{array}{cccccccccc}
(Q(\Theta(\mathbb{Z}_{n}))+nI)_{n\times n} & & & J_{n\times n}
\\
\\
J^T_{n\times n} & & & ((2n-2)I+J)_{n\times n}
\end{array}
\right).
\end{equation}

Using Equation (\ref{Di}) we find that the signless Laplacian matrix of $\Theta(D_{n})$ depends on the signless Laplacian matrix of $\Theta(\mathbb{Z}_{n})$.
In the previous section we had determined the eigenvalues of $Q(\Theta(\mathbb{Z}_{n}))$ for $n\in \{p^m,pq\}$.
We will use those in this section to find the eigenvalues of $Q(\Theta(D_{n}))$ for $n\in \{p^m,pq\}$.

\begin{theorem}
	\label{Dpq}
If $n=pq$, then the eigenvalues of $Q(\Theta(\mbox D_{n}))$  are $2(n-1)$ with multiplicity $2n-\varphi(n)-1$, $2n-\varphi(n)$ with multiplicity $\varphi(n)-1$, and $3n-\varphi(n)-1\pm \sqrt{n^2+2\varphi(n)n-\varphi(n)^2-2n+1}$ each with multiplicity $1$.
\end{theorem}

\begin{proof}
	If $n=pq$, using Equation (\ref{Equation 1}) of Theorem \ref{Spec1} and Equation (\ref{Di}), we find that the signless Laplacian matrix of $\Theta(\mbox D_{n})$ is of the following form:
	\begin{equation*}
	\begin{split}
	Q(\Theta(\mbox D_{n}))&=
	\left(\begin{array}{cccccccc}
	((2n-2)I+J)_{n-\varphi(n)\times (n-\varphi(n))} && J_{(n-\varphi(n)\times \varphi(n))} && J_{(n-\varphi(n)\times n)}\\
	\\
	J_{\varphi(n)\times (n-\varphi(n))} && (2n-\varphi(n)) I_{\varphi(n)\times \varphi(n)} && J_{\varphi(n)\times n}\\
	\\
	J_{n\times (n-\varphi(n))} && J_{n\times \varphi(n)} && ((2n-2)I+J)_{n\times n}
	\end{array}
	\right).
	\end{split}
	\end{equation*}
	
	If we consider the matrix $Q-(2n-2)I$, we find that it  has $2n-\varphi(n)$ identical rows. Thus, $2n-2$ is an eigenvalue of $Q$ with multiplicity at least $2n-\varphi(n)-1$.
	We also note that  $Q-(2n-\varphi(n))I$  has $\varphi(n)$ identical rows, which makes us conclude that $2n-\varphi(n)$ is an eigenvalue of $Q$ with multiplicity at least $\varphi(n)-1$.
	
	We partition the vertex set $V$ of $\Theta(\mbox D_n)$ in the following way:
	$V_1=\{1\}\cup \{r^i: \gcd(i,n)\neq 1\}$, $V_2=\{r^i: \gcd(i,n)= 1\}$ and $V_3=\{sr^i: 0\le i\le n-1\}$.
	Each vertex $v\in V_1$ has $n-\varphi(n)-1$ neighbors in $V_1$, $\varphi(n)$ neighbors in $V_2$ and $n$ neighbors in $V_3$.
	Each vertex $v\in V_2$ has $n-\varphi(n)$ neighbors in $V_1$, $0$ neighbors in $V_2$ and $n$ neighbors in $V_3$, and each vertex $v\in V_3$ has $n-\varphi(n)$ neighbors in $V_1$, $\varphi(n)$ neighbors in $V_2$ and $n-1$ neighbors in $V_3$.
	Hence, the partition is an equitable partition, and we call it $\pi$.
	Using Equation  (\ref{EqP}), the equitable quotient matrix corresponding to $\pi$ is given by:
	
	\[
	Q_{\pi}=
	\left(\begin{array}{ccccccc}
	3n-2-\varphi(n) & & \varphi(n) & & n\\
	\\
	n-\varphi(n) & & 2n-\varphi(n) & & n\\
	\\
	n-\varphi(n) & & \varphi(n) & & 3n-2
	\end{array}\right).\]
	The characteristic polynomial of $Q_{\pi}$ is 
	\begin{flalign*}
	\Lambda(x)&=x^3 + (2\varphi(n) - 8n + 4)x^2 + (2\varphi(n)^2 - 12\varphi(n)n + 20n^2 + 6\varphi(n) - 20n + 4)x
	\\& - 4\varphi(n)^2n + 16\varphi(n)n^2 - 16n^3 + 4\varphi(n)^2 - 20\varphi(n)n + 24n^2 + 4\varphi(n) - 8n.
	\end{flalign*}
	On solving, we find that solutions of $\Lambda(x)=0$ are 
	$2(n-1)$ with multiplicity $1$ and $3n-\varphi(n)-1\pm \sqrt{n^2+2\varphi(n)n-\varphi(n)^2-2n+1}$ each with multiplicity $1$.
	We further note that $\Lambda(2n-\varphi(n))\neq 0$, as otherwise it would imply 
	\begin{flalign*}
	2n-\varphi(n)&= 3n-\varphi(n)-1\pm\sqrt{n^2+2\varphi(n)n-\varphi(n)^2-2n+1}
	\\
	\text{which implies } (n-1)^2&=n^2+2n\varphi(n)-\varphi(n)^2-2n+1
	\\
	\text{which implies } 2n\varphi(n)-\varphi(n)^2&=0
	\\
	\text{which implies } (2n-\varphi(n))\varphi(n)&=0
	\text{ which is false for } n=pq. 
	\end{flalign*}
	Again, 	we further note that $2n-2\neq 3n-\varphi(n)-1\pm\sqrt{n^2+2\varphi(n)n-\varphi(n)^2-2n+1}$, as otherwise it would imply 
	\begin{flalign*}
		2n-2=3n-\varphi(n)-1\pm\sqrt{n^2+2\varphi(n)n-\varphi(n)^2-2n+1}\\
		\text{which implies } \biggl(n+1-\varphi(n)\biggr)^2&=n^2+2\varphi(n)n-\varphi(n)^2-2n+1\\
		\text{which implies } 2\varphi(n)^2+4n-2\varphi(n)-4n\varphi(n)&=0\\
		\text{which implies } \biggl(\varphi(n)-2n\biggr)\biggl(\varphi(n)-1\biggr)&=0\\
		\text{which implies either } \varphi(n)=2n \text{ or } \varphi(n)=1 
		\text{ which are both false.}
	\end{flalign*}

	We thus conclude that the eigenvalues $3n-\varphi(n)-1\pm \sqrt{n^2+2\varphi(n)n-\varphi(n)^2-2n+1}$ of $Q_{\pi}$ are distinct from both $2n-\varphi(n)$ and $2(n-1)$ for $n=pq$.
	Using  Lemma $5.1$ of \cite{banerjee2020signless}, we find that $3n-\varphi(n)-1\pm \sqrt{n^2+2\varphi(n)n-\varphi(n)^2-2n+1}$ are eigenvalues of $Q$ each with multiplicity $1$.
	Thus, the eigenvalues of $Q$  are $2(n-1)$ with multiplicity $2n-\varphi(n)-1$, $2n-\varphi(n)$ with multiplicity $\varphi(n)-1$ and $3n-\varphi(n)-1\pm \sqrt{n^2+2\varphi(n)n-\varphi(n)^2-2n+1}$ each with multiplicity $1$.

\end{proof}

\begin{proposition}
	If $n=p$, then the eigenvalues of $Q(\Theta(\mbox{D}_{n}))$ are  $2(n-1)$ with multiplicity $1$ and $n-2$ with multiplicity $n-1$.
\end{proposition}

\begin{proof}
	If $n=p$, then using Corollary \ref{Th31}, $\Theta(\mbox{D}_n)$ is complete.
	Thus, $Q=(n-2)I +J$. Using Theorem \ref{Useit}, the eigenvalues of $Q$ are  $2(n-1)$ with multiplicity $1$ and $n-2$ with multiplicity $n-1$.
\end{proof}

\begin{theorem}
	If $n=p^m$ where $m\ge 2$, then the eigenvalues of $Q(\Theta(\mbox{D}_{n}))$ are  $2n-2$  with multiplicity  $n+p-1$, $p+n$ with multiplicity  $n-p-1$, and $2n+p-1\pm \sqrt{2n^2-2n-p^2+1}$ each with multiplicity $1$.
\end{theorem}

\begin{proof}
If $n=p^m$, using Equation (\ref{Equation 1}) of Theorem \ref{Spec2} and Equation (\ref{Di}) we find that $Q=Q(\Theta(\mbox D_{n}))$ is of the following form:
	\begin{equation*}
	\begin{split}
	Q&=
	\left(\begin{array}{ccccc}
	((2n-2)I+J)_{p\times p} && J_{p\times (n-p)} && J_{p\times n}\\
	\\
	J_{(n-p)\times p} && (p+n) I_{(n-p)\times (n-p)} && J_{(n-p)\times n}\\
	\\
	J_{n\times p} && J_{n\times (n-p)} && ((2n-2)I+J)_{n\times n}
	\end{array}
	\right).
	\end{split}
	\end{equation*}
	
	We note that $Q-(2n-2)I$ has $n+p$ identical rows, and hence  $2n-2$ is an eigenvalue of $Q$ with multiplicity at least $n+p-1$.
	Similarly,  $Q-(p+n)I$ has $n-p$ identical rows, and hence $p+n$ is an eigenvalue of $Q$ with multiplicity at least $n-p-1$.
	
	We now partition the vertex set $V$ of $\Theta(\mbox D_{p^m})$ as $V=V_1\cup V_2\cup V_3$ where 
	$$V_1=\{1,r^{p^{m-1}},r^{2(p^{m-1})},\ldots, r^{(p-1)(p^{m-1})}\},$$ $V_2=\{r^i: 0\le i\le n-1\}\setminus V_1$, and $V_3=\{sr^i: 0\le i\le n\}$.
	Each vertex $v\in V_1$ has $p-1$ neighbors in $V_1$, $n-p$ neighbors in $V_2$ and $n$ neighbors in $V_3$.
	Each vertex $v\in V_2$ has $p$ neighbors in $V_1$, $0$ neighbors in $V_2$ and $n$ neighbors in $V_3$ and each vertex $v\in V_3$ has $p$ neighbors in $V_1$, $n-p$ neighbors in $V_2$ and $n-1$ neighbors in $V_3$.
	Hence, the partition is an equitable partition, and we call it $\pi$.
	
	The equitable quotient matrix of $Q$ corresponding to $\pi$ is given by:
	\[Q_{\pi}=\left(\begin{array}{ccccccc}
	2(n-1)+p & & n-p & & n\\
	p & & n+p & & n\\ 
	p & &n-p && 3n-2
		\end{array}\right).
		\]
	The characteristic polynomial of $Q_{\pi}$ is 
	\begin{flalign*}
	\Lambda(x)&=x^3 + (-6n - 2p + 4)x^2 + (10n^2 + 8np + 2p^2 - 14n - 6p + 4)x 
	\\& -4n^3 - 8n^2p - 4np^2 + 8n^2 + 12np + 4p^2 - 4n - 4p.
	\end{flalign*}	
	We find that the solutions of $\Lambda(x)=0$ are $2(n-1)$ and $2n+p-1\pm \sqrt{2n^2-2n-p^2+1}$.
	We further note that, $\Lambda(n+p)\neq 0$, as otherwise it would imply, 
	\begin{flalign*}
	n+p&= 2n+p-1\pm \sqrt{2n^2-2n-p^2+1}
	\\
	\text{which implies } -n+1&=\pm \sqrt{2n^2-2n-p^2+1}.
	\\
	\text{which implies } (n-1)^2&=2n^2-2n-p^2+1
	\\
	\text{which implies } n^2&=p^2 	\text{ which is false}. 
	\end{flalign*}	
		
	We further note that $2(n-1)\neq 2n+p-1\pm \sqrt{2n^2-2n-p^2+1}$, as otherwise
	\begin{flalign*}
		2(n-1)&=2n+p-1\pm \sqrt{2n^2-2n-p^2+1}\\
			\text{which implies } -p-1&=\pm  \sqrt{2n^2-2n-p^2+1}\\
			\text{which implies } (p+1)^2&=2n^2-2n-p^2+1\\
			\text{which implies } 2p(p+1)&=2n(n-1)\\
			\text{which implies } p+1&= p^{m-1}(p^m-1)\\
			\text{which implies } \frac{p^m-1}{p+1}&=\frac{1}{p^{m-1}}<1 \text{ which is false}.
	\end{flalign*}
		
	Thus, the eigenvalues $2n+p-1\pm \sqrt{2n^2-2n-p^2+1}$ of $Q_\pi$ are different from both $n+p$ and $2(n-1)$ for $n=p^m$.
	Using   \cite[Lemma $5.1$]{banerjee2020signless}, we conclude that $2n+p-1\pm \sqrt{2n^2-2n-p^2+1}$ are eigenvalues of $Q$ each with multiplicity  $1$. Thus, the eigenvalues of $Q$ are  $2n-2$  with multiplicity  $n+p-1$, $p+n$ with multiplicity   $n-p-1$, and $2n+p-1\pm \sqrt{2n^2-2n-p^2+1}$ each with multiplicity $1$.

\end{proof}

\printbibliography

\end{document}